\documentclass[11pt,a4paper]{article}
\usepackage[utf8]{inputenc}			
\usepackage[T1]{fontenc}            
\usepackage{a4wide}
\usepackage{amssymb,amsmath,amsthm,bm}
\usepackage{mathtools,lmodern,enumitem}
\usepackage[english]{babel}

\usepackage{color}
\usepackage{tikz}
\usepackage{amsfonts,epsf,amsmath,comment}
\usepackage{lineno}
\usepackage{comment}

\usepackage{subfig}
\usepackage{epic}
\usepackage{amssymb}
\usepackage{latexsym}
\usetikzlibrary{decorations.markings}
\usetikzlibrary{decorations.pathreplacing}

\usetikzlibrary{arrows}
\newcommand{\R}{R}
\newcommand{\B}{B}
\newcommand{\K}{\mathcal K}
\newcommand{\T}{\mathcal T}
\newcommand{\dR}{d_{\R}}

\usetikzlibrary{arrows}
\usepackage{graphicx}

\usepackage{hyperref}

\usepackage{comment}

\newtheorem{theorem}{Theorem}
\newtheorem{corollary}[theorem]{Corollary}
\newtheorem{lemma}[theorem]{Lemma}
\newtheorem{proposition}[theorem]{Proposition}

\newtheorem{rem}[theorem]{Remark}
\newtheorem{conjecture}[theorem]{Conjecture}

\newtheorem{problem}[theorem]{Problem}

\newtheorem{definition}[theorem]{Definition}

\newcommand{\dedicatory}[1]{
  \begin{center}
    \vspace{-1em}
    \textit{#1}
    \vspace{1.5em}
  \end{center}
}

\def\marrow{{\marginpar[\hfill$\longrightarrow$]{$\longleftarrow$}}}

\def\zoltan#1{{\sc Zolt\'an says: }{\textcolor{red}{\marrow\sf #1}}}

\newcommand{\mailto}[1]{\href{mailto: #1}{#1}}

\title{Two-coloring cubic graphs with small monochromatic components, but without singletons}
\author{János Barát\thanks{HUN-REN Alfr\'ed R\'enyi Institute of Mathematics, Budapest, Hungary and University of Pannonia, Veszprém, Hungary, \mailto{barat@mik.uni-pannon.hu}
}
\and
Zoltán L. Blázsik\thanks{Bolyai Institute, University of Szeged, Aradi v\'ertan\'uk tere 1, 6720
Szeged, Hungary. \newline University of Johannesburg Auckland Park, 2006 South Africa, \mailto{blazsik@server.math.u-szeged.hu}
}
}

\begin{document}
\thispagestyle{empty}
\maketitle

\dedicatory{Dedicated to Professor Tibor Szabó on the occasion of his 60th birthday}

\begin{abstract}
We combine two coloring aspects that work in opposite directions. 
One can 2-color the vertices of a cubic graph such that each monochromatic component is very small.
One can also 2-color the vertices of a cubic graph such that each monochromatic component has degree at least 1.
As an intended tool for solving a special case of Wegner's conjecture, Thomassen formulated a conjecture that combined the two previous properties.
This led to the concept of a crumby coloring.
However it turned out that there are cubic graphs without such coloring.
Here we try to see what natural relaxations of the original concept might hold for each cubic graph.

We show there exists a constant $c$ such that every cubic graph has a vertex 2-coloring such that every monochromatic component has at least 2 and at most $c$ vertices.
We also prove an unbalanced version, which is the natural relaxation of the crumby coloring.

\end{abstract}


\section{Introduction}

Thomassen \cite{ct17} formulated the following conjecture: Every $3$-connected cubic graph has a red-blue vertex coloring such that the blue subgraph has maximum degree at most $1$
(that is, it consists of a matching and some isolated vertices) and the red
subgraph has minimum degree at least $1$ and contains no $3$-edge path.
Since all monochromatic components are small in this coloring and there is a certain irregularity,
Barát \cite{bj} coined the name {\em crumby coloring} for this concept.
Bellitto, Klimo\v sová, Merker, Witkowski and Yuditsky \cite{counter}  constructed cubic and subcubic graphs that do not admit  crumby colorings.
We still suspect that similar colorings exist for all cubic graphs.
Let us relax the condition on the size of a red component. 
A red-blue vertex coloring is $\ell$\emph{-crumby} if we require the same for the blue components as before and the red connected components must have size at least $2$ and do not have paths of length~$\ell$. 

There are several notions that are closely related to crumby colorings.
A \emph{$k$-bisection} of a graph $G$ is a partition of its vertex set $V(G)$ into two sets $V_1$ and $V_2$ such that $||V_1|-|V_2| |\le 1$ and every connected component of $G[V_i]$ has at most $k$ vertices, where $G[V_i]$ denotes the subgraph of $G$ induced by $V_i$ (for $i=1,2$). 
Cui and Liu~\cite{cl} proved that every subcubic graph admits a 3-bisection $(V_1,V_2)$ with an additional property that every connected component of $G[V_i]$ ($i=1,2$) is acyclic. This extended the result of Esperet, Mazzuoccolo and Tarsi~\cite{emt} who showed that every cubic graph admits such a 3-bisection.

Another closely related notion is \emph{relaxed two-coloring}.  In a $(C_1, C_2)$-relaxed coloring
of a graph $G$ every monochromatic component induced by vertices of
the first (second) color is of order at most $C_1$ ($C_2$, respectively). 
Berke and Szabó~\cite{bsz} were mostly concerned with $(1, C)$-relaxed colorings, in other words when is it possible to break up a graph into small components with the removal of an independent set.
They proved that every graph of maximum degree at most three can be
$(1, 22)$-relaxed colored. 

A third similar notion investigated by Cranston and Yancey \cite{cranston} is an $(I,F_3)$-coloring. That is, the blue graph consists of singletons and the red graph is a forest, where each component has at most 3 vertices. However, the red forest might have singletons. This is very similar to a crumby coloring. 


Our first result concerns two-colorings of cubic graphs such that every monochromatic component has order at least $2$. The aim is to give an upper bound on the size of the monochromatic components.

\begin{theorem}\label{thm:pm}
    If $G$ is a cubic graph that has a perfect matching, then there is a red-blue vertex coloring of the vertices such that every monochromatic component has at least $2$ and at most $12$ vertices. 
\end{theorem}

If $G$ does not have a perfect matching, then we can use the famous Edmonds-Gallai theorem to obtain the following result.

\begin{theorem} \label{t:bal_EG}
If $G$ is a cubic graph, then there is a red--blue vertex coloring of $G$ such that every monochromatic component has at least $2$ and at most $13$ vertices.
\end{theorem}

It is known that there are some subcubic graphs for which a $3$-crumby coloring does not exist. In principle, it might happen that a $4$-crumby coloring always exists for cubic graphs. However, the existence of an $\ell$-crumby coloring for any $\ell$ is not known. Our main result proves this for any subcubic graph and $\ell=88$.

\begin{theorem}\label{thm:main88}
Every subcubic graph $G$ has an $88$-crumby coloring. In other words, there exists a red--blue coloring such that
$  \Delta(G[\B])\le 1,
  \delta(G[\R])\ge 1,
$
and every path in $G[\R]$ has at most $88$ edges, where $R$ is the set of red and $B$ is the set of blue vertices.
\end{theorem}

The paper is organised as follows. In Section 2, we focus on the balanced version, where the existence of  monochromatic singletons are forbidden. In Section 3, we prove our main result about the existence of an $88$-crumby coloring of any subcubic graph. 
We finish our paper with some open problems.


\section{Small components without singletons}

There is one well-known key ingredient, the Edmonds-Gallai decomposition, that we use in our proof just like Berke and Szabó did in \cite{bsz}. 
A component $C$ of a graph is {\it hypomatchable} if removing any vertex $v$ of $C$ has a perfect matching. 

\begin{theorem}[Edmonds-Gallai decomposition \cite{edm,gal}]\label{e-g}
Let $G$ be a graph and let $A\subseteq V(G)$ be the collection of all vertices $v$ such that there exists a maximum size matching which does not cover $v$.
Set $B=N(A)$ and $C=V(G)\setminus (A \cup B)$. Now
\begin{enumerate}
    \item[$(i)$] Every odd component $O$ of $G-B$ is hypomatchable and $V(O)\subseteq A$.
    \item[$(ii)$] Every even component $Q$ of $G-B$ has a perfect matching and $V(Q)\subseteq C$. 
    \item[$(iii)$] For every $X\subseteq B$, the set $N(X)$ contains vertices in more than $|X|$ odd components of $G-B$.
\end{enumerate}
\end{theorem}

We also need another nice result of Haxell, Szabó and Tardos~\cite{hszt} which we apply in the next section. We would like to emphasize that in the next result, singleton components are permitted. 

\begin{theorem}[Haxell, Szab\'o, Tardos \cite{hszt}]\label{hst4}
 If $G$ is a graph with maximum degree at most $4$, then there is a $2$-coloring of the vertices such that every monochromatic component
 has at most $6$ vertices.
\end{theorem}

Inspired by the Haxell, Szabó and Tardos result, we would like to focus on two-colorings such that all monochromatic components are quite small, but none of them are singletons. We start with some observations in which the monochromatic components are measured by their radii instead of their size. The main idea is to use large matchings and contract the matching pairs which will allow us to use Theorem~\ref{hst4} and after splitting the matching pairs there should not be any singletons because of the matching pairs.



\begin{lemma}\label{3con}
 If $G$ is a cubic $2$-edge-connected graph, then there is a red-blue coloring of the vertices of $G$ such that every monochromatic component
 has radius at least $1$ and at most $6$.
\end{lemma}

\begin{proof}
 Petersen's theorem implies that $G$ has a perfect matching $M$.
 Let us contract all edges in $M$ such that a $4$-regular graph $H$ arises.
 Now $H$ has a red-blue coloring with all components having at most $6$ vertices by Theorem~\ref{hst4}.
 We open up the vertices now to get back $G$.
 In that way a monochromatic component of $H$ with at most $6$ vertices can grow to the double.
 However, it will have radius at most $6$. 
\end{proof}



Theorem~\ref{thm:pm} immediately follows from Lemma~\ref{3con}. Next we show that $12$ cannot be decreased below $7$ in Theorem~\ref{thm:pm}.

\begin{proposition}\label{prop:hea}
The Heawood graph has a perfect matching, and every red--blue coloring of its vertices with no monochromatic singleton contains a monochromatic component on at least $7$ vertices. Moreover, this bound is sharp for the Heawood graph.
\end{proposition}

\begin{proof}
Let $H$ be the Heawood graph, regarded as the incidence graph of the Fano plane. Thus one bipartition class consists of the seven points and the other of the seven lines. In particular, $H$ is cubic and bipartite, and hence it has a perfect matching.

Suppose that $H$ has a red--blue coloring in which every monochromatic component has order between $2$ and $6$. Let $r$ be the number of red vertices, and, by interchanging the colors if necessary, assume that $r\ge7$.

We first show that $r\le8$. Since $H$ has girth $6$, every connected subgraph on at most $6$ vertices is either a tree or a $6$-cycle. Hence a monochromatic component on $s\le6$ vertices has at least $s+2$ edges leaving it, unless it is a $6$-cycle, in which case it has $6$ edges leaving it. If $r\ge9$, the red vertices form at least two components, and it follows that at least $11$ red--blue edges leave the red vertices. On the other hand, every blue vertex has a blue neighbor, and therefore is incident with at most two red--blue edges. Hence $|E(R,B)|\le 2(14-r)\le10$ provides a contradiction. Thus $r\in\{7,8\}$.

Let $p$ denote the number of red points. By the point--line duality of the Fano plane, we may assume that the number of red points is at most the number of red lines. Hence $p\le3$ if $r=7$, and $p\le4$ if $r=8$. Notice also that every red line contains a red point and every red point lies on a red line, since monochromatic singletons are forbidden. The analogous statement holds for blue.

Suppose first that $r=7$. We have $p\ge2$. If $p=2$, then there are five red lines. The five lines incident with at least one of the two red points are therefore precisely the red lines. In particular, the line through the two red points is red, so all seven red vertices belong to one component, a contradiction.

Thus $p=3$. If the three red points are collinear, their common line must be red, since otherwise it would be an isolated blue vertex. The common red line connects the three red points, and every other red line is incident with one of them. Hence again all seven red vertices form one component.

Suppose therefore that the three red points are non-collinear. There are three lines containing two red points, three lines containing exactly one red point, and one line containing no red point. The latter line must be blue. If at least two of the three lines containing two red points were red, then the three red points, and hence all red vertices, would be connected. Consequently exactly one of these three lines is red, and all three lines containing exactly one red point are red. These three lines are concurrent at the unique point $x$ which is outside of the three lines determined by a pair of red points. The point $x$ is blue, while all three lines through it are red, so $x$ is an isolated blue vertex, a contradiction.

It remains to consider $r=8$. Here $p\ge3$, since two red points are incident with only five lines, whereas there are at least six red lines if $p\le2$.

Suppose first that $p=3$. If the three red points are collinear, their common line must be red and all eight red vertices are connected, as above. Thus the three red points are non-collinear. The unique line disjoint from them is blue. Since there are five red lines, there is only one further blue line. The fourth blue point must have a blue neighbor, so this second blue line is one of the three lines containing exactly one red point. Consequently all three lines containing two red points are red. These three lines connect the three red points, and every remaining red line is incident with one of them. Thus all eight red vertices are connected, a contradiction.

Finally let $p=4$. There are four red points and four red lines, and three blue points.

If the three blue points are collinear on a line $L$, then $L$ must be blue. Every other line contains exactly two red points. Thus the four red lines correspond to four edges of the complete graph $K_4$ on the four red points. Since no red point is isolated, this four-edge graph has minimum degree at least one, and hence it is connected. Therefore all eight red vertices are connected, a contradiction.

Suppose instead that the three blue points are non-collinear. There is a unique line $L$ disjoint from them; it contains three red points and must be red. Let $x$ be the fourth red point. The three lines through $x$ each contain one of the three red points of $L$ and one blue point. If none of these three lines is red, then $x$ is an isolated red vertex. If at least one is red, then $x$ is joined to the red component containing $L$. Hence all four red points lie in one red component, and every red line is incident with this component. Thus all eight red vertices are connected. This is again a contradiction.

We conclude that every red--blue coloring of the Heawood graph without monochromatic singletons has a monochromatic component of order at least $7$.

\medskip

It remains to show that $7$ is attainable. Represent the points of the Fano plane by the non-zero vectors of $\mathbb F_2^3$, and let $e_1,e_2,e_3$ be a basis. Color the points $e_1,e_2, e_3$ red, together with the four lines $\{e_1,e_2,e_1+e_2\},~
\{e_2,e_3,e_2+e_3\},~\{e_3,e_1,e_3+e_1\},~
\{e_1,e_2+e_3,e_1+e_2+e_3\}.$

\begin{center}
\includegraphics{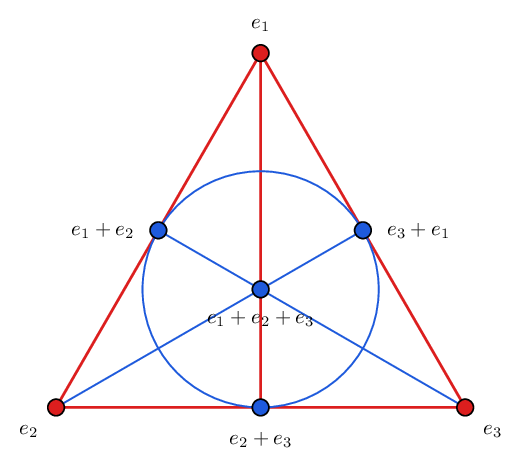}\end{center}


Color all remaining vertices blue. The red induced subgraph is connected on seven vertices, and so is the blue induced subgraph. Thus both monochromatic components have order exactly $7$.

Therefore the Heawood graph shows that any universal bound for cubic graphs with a perfect matching is at least $7$.
\end{proof}

A natural question arises: what is the minimum value $r$ such that any cubic graph has a two-coloring such that every monochromatic component has radius at least $1$ and at most $r$. It is straightforward to prove that the Petersen graph shows that $r>1$ is necessary.


\begin{conjecture}\label{conj}
If $G$ is a cubic $2$-edge-connected graph, then there is a red-blue coloring of the vertices of $G$ such that every monochromatic component
 has radius at least $1$ and at most $2$. This is true up to 18 vertices.
\end{conjecture}

In the next step, we handle cubic graphs without a perfect matching. Let us recall the statement.

\medskip

\noindent {\bf Theorem~\ref{t:bal_EG}.} If $G$ is a cubic graph, then there is a red--blue vertex coloring of $G$ such that every monochromatic component has at least $2$ and at most $13$ vertices.

\medskip

\noindent {\it Proof of Theorem~\ref{t:bal_EG}.~} If $G$ has a perfect matching, then the statement follows from Theorem~\ref{thm:pm}. Suppose that $G$ has no perfect matching, and consider the Edmonds--Gallai decomposition $V(G)=A\cup B\cup C$.

Let $\mathcal O$ be the set of odd components of $G-B$, and form the bipartite graph with bipartition $(\mathcal O,B)$ in the natural way. By Theorem~\ref{e-g}\emph{(iii)}, it has a matching saturating $B$. For $O\in\mathcal O$, let
$d_B(O)=|E(O,B)|$.
Since $G$ is cubic and $|V(O)|$ is odd, $d_B(O)$ is odd for each $O\in\mathcal{O}$. Among all matchings saturating $B$, choose one, say $M$, which saturates as many components $O$ with $d_B(O)\ge3$ as possible. We claim that every such component is saturated.

Suppose otherwise, and let $O_0$ be an unsaturated component with $d_B(O_0)\ge3$. Perform an alternating search from $O_0$. 
This means selecting an unsaturated edge from $O_0$ to $b\in B$, then selecting a saturated edge of the matching from $b$, then again selecting an unsaturated edge from the next odd component etc. It corresponds to an alternating path, if we contracted the odd components to single vertices. 
Let $\mathcal S$ be the set of reached odd components and let $T\subseteq B$ be the set of reached vertices of $B$. Since $M$ saturates $B$, $|T|=|\mathcal S|-1$ follows.

If some component $O\in\mathcal S\setminus{O_0}$ had $d_B(O)=1$, then switching $M$ along an alternating sequence from $O_0$ to $O$ would still saturate $B$, while saturating one more component with at least 3 edges to $B$ gives a contradiction. Hence every component in $\mathcal S$ has at least three edges to $B$. Moreover, by the definition of the alternating search, all their neighbors in $B$ belong to $T$. Therefore $3|\mathcal S|
 \le |E(\mathcal S,T)|
 \le 3|T|
 =3(|\mathcal S|-1),$ a contradiction. Thus every unsaturated odd component $O$ satisfies
$
d_B(O)=1$.

For every odd component saturated by $M$, choose the corresponding edge $bv$, with $b\in B$ and $v\in O$. Since $O$ is hypomatchable, $O-v$ has a perfect matching; add these edges and $bv$ to $M$. Also add a perfect matching from every even component of $G-B$.

Now let $O$ be an unsaturated odd component. Let us denote the unique vertex adjacent to a vertex $b$ from $B$ with $w$. Delete $wb$. Since $O-w$ has a perfect matching $P$ and $w$ has two neighbors in $O$, choose one such neighbor $z$. Let $zu\in P$, and contract $wz$. Then
$
(P\setminus\{zu\})\cup\{(wz)u\}
$
is a perfect matching of the contracted copy of $O$.

After doing this for every unsaturated odd component, we obtain a graph with a perfect matching. Contract all edges of this perfect matching and call the resulting graph $H$. We have
$
\Delta(H)\le4.
$

Indeed, an ordinary matching edge gives a contracted vertex of degree at most four. At an exceptional vertex, the preliminary contraction $wz$ has degree at most three after the edge $wb$ has been deleted, and contracting its matching edge therefore again gives degree at most four.

By Theorem~\ref{hst4}, $H$ has a red--blue coloring in which every monochromatic component has at most six vertices. Lift this coloring through all contractions. Every ordinary vertex of $H$ becomes two adjacent vertices of the same color, while in each unsaturated odd component one exceptional vertex becomes three connected vertices of the same color. Hence there are no monochromatic singleton components.

Moreover, after the deleted edges $wb$ are removed, every unsaturated odd component is a connected component of the graph and contains exactly one exceptional contracted vertex. Consequently every monochromatic component has at most $2\cdot6+1=13$ vertices.

Finally, for each unsaturated odd component $O$, interchange red and blue on all vertices of $O$ if necessary, so that the endpoints of its unique deleted edge $wb$ have different colors. These color changes can be made independently, since the edge $wb$ is the only edge joining $O$ to the rest of the graph. Restore all deleted edges. They are now bichromatic, so no monochromatic components are joined.
Thus every monochromatic component has at least $2$ and at most $13$ vertices. \hfill \qed

\section{Generalized crumby}

Berke and Szabó \cite{bsz} study relaxations of proper two-colorings such that the order of the induced monochromatic components in one (or both) of the color classes is bounded by a constant. A coloring of a graph $G$ is called $(C_1, C_2)$-relaxed if every monochromatic component induced by vertices of the first (second) color is of order at most $C_1$ ($C_2$, respectively). They 
find a $(1, 22)$-relaxed coloring of any graph of maximum degree 3. However, their coloring allows singleton vertices in both colors. This is excluded in the generalized $\ell$-crumby coloring setup.

\begin{definition}\label{def:crumby}
    For an arbitrary graph $G$ and $\ell\ge 1$, a red--blue coloring of the vertices of $G$ ($V(G)=R~\dot\cup~B$) is an $\ell$-crumby coloring if the following three conditions hold: 
    \begin{enumerate}[label=(C\arabic*)]
  \item $\Delta(G[\B])\le 1$;
  \item $\delta(G[\R])\ge 1$;
  \item $G[\R]$ contains no path with $\ell$ edges.
\end{enumerate}
\end{definition}

Let us call a red--blue coloring which satisfies $(C1)$ and $(C2)$ a \emph{locally feasible} coloring. For clarity reasons, we divided the argument into smaller steps so that eventually the actual proof should be easily understandable. 

\begin{lemma}\label{lem:existence-local}
Every graph has a locally feasible coloring.
\end{lemma}

\begin{proof}
Color every vertex belonging to a non-trivial connected component red and
every isolated vertex blue. Then the blue graph is independent and every red vertex has a red neighbor.

One can also consider a maximal matching. Color the vertices of matching edges red, and color the remaining vertices blue.
\end{proof}

The starting point of our coloring will be an \emph{extremal} feasible coloring.

\begin{definition}\label{def:ex}
    A locally feasible coloring is \emph{extremal} if the following two conditions hold:
    \begin{enumerate}[label=(E\arabic*)]
  \item maximizes the number $|E(\R,\B)|$ of red--blue edges;
  \item subject to (E1), maximize $|\R|$.
\end{enumerate}    
\end{definition}

Fix from now on a subcubic graph $G$ and consider an extremal coloring $\mathcal{C}$ of $V(G)$. 

\begin{definition}\label{def:red-core}
The \emph{red core} of an extremal coloring is
  $\K
  =G\bigl[\{v\in\R:d_{G[R]}(v)\ge 2\}\bigr]$.
A vertex of $\K$ is called \emph{internal} if it has degree two in $\K$, otherwise it is called \emph{external}.
A red neighbor $x$ of a core vertex $v$ is a \emph{private red leaf of $v$}
if $d_{G[R]}(x)=1$ and $N_{G[\R]}(x)=\{v\}$.
\end{definition}

The proof of our main result uses a well-known result of Haxell from \cite{Hax}.

\begin{theorem}[Haxell]\label{thm:Haxell} 
    Let $k$ be a positive integer, let $H$ be a graph with maximum degree at most $k$, and let $V(H)=V_1\cup \dots \cup V_n$ be a partition of the vertex set of $H$. If $|V_i|\ge 2k$ for each $i$, then $H$ has an independent set $\{ v_1,\dots,v_n\}$ where $v_i\in V_i$ for each $i$.
\end{theorem}

Now, we give an outline of the proof. The extremality of the initial coloring infers some strong properties of the red core. 
Then we would like to perform some local recoloring at the internal vertices which would shorten the longest path of the corresponding red component while maintaining the extremality of the coloring. However, it may happen that after this recoloring some red components merge together but the placement of these merges can be controlled. Moreover these conflicts can be encoded in a conflict graph for which we can apply the result of Haxell to find a simultaneous non-conflicting choice of local recolorings.
Finally, we show that with these steps we can achieve that the longest path in any red component have length at most $88$.



The first lemma is the structural heart of the proof. 

\begin{lemma}\label{lem:maxcut-normal-form}
Let $G$ be a subcubic graph and let $V(G)=\R~\dot\cup~\B$ be extremal. Then:
\begin{enumerate}[label=(\roman*)]
  \item $\Delta(\K)\le 2$;
  \item every internal vertex of $\K$ has degree three in $G$;
  \item if $v$ is internal with core-neighbors $p,q$, then its third
  neighbor is a private red leaf $x$.
\end{enumerate}
In particular, no internal core vertex has a blue neighbor.
\end{lemma}

\begin{proof}
Suppose first that $v$ has three neighbors in $\K$.  Recolor $v$ blue.  Each
of its three red neighbors belongs to the core and therefore retains at least
one red neighbor after the recoloring.  The vertex $v$ has no blue neighbor.
Thus local feasibility is preserved, while all three edges incident with
$v$ become red--blue edges.  This contradicts \emph{(E1)}.  Hence
$\Delta(\K)\le 2$.

Let now $v$ be internal, with core-neighbors $p$ and $q$.  If $d_G(v)=2$,
then recoloring $v$ blue is locally feasible: both $p$ and $q$ retain a red
neighbor, and $v$ has blue degree zero.  The two edges $vp$ and $vq$ enter
the red--blue cut, contradicting \emph{(E1)}. Therefore we can assume $d_G(v)=3$. 

Let $x$ be the third neighbor of $v$.  If $x$ is red, then $x\notin V(\K)$,
because $v$ already has exactly two neighbors in $\K$.  Since $x$ is red
and adjacent to $v$, it follows that $\dR(x)=1$.  Thus $x$ is a private red
leaf of $v$.

It remains to rule out the case that $x$ is blue.  First suppose that $x$ is
isolated in $G[\B]$.  Recoloring $v$ blue remains locally feasible: $p$ and
$q$ retain red neighbors, while $v$ and $x$ form at most a blue edge.  The
edges $vp,vq$ enter the cut and $vx$ leaves it, for an increase of one,
contradicting \emph{(E1)}.

Now suppose $x$ belongs to a blue edge component $xw$.  Suppose first that $d_G(x)=3$,
and let $z$ be the third neighbor of $x$.  Since $x$ already has the blue
neighbor $w$, the blue degree bound forces $z$ to be red.  Exchange the
colors of $v$ and $x$.  The new blue vertex $v$ has only red neighbors, and
the new red vertex $x$ is supported by $z$.  The red vertices $p$ and $q$
retain red neighbors, and $w$ loses its blue neighbor.  Hence the exchange
is locally feasible.  The edges $vp,vq,xw$ enter the cut, the edge $xz$
leaves it, and $xv$ remains bichromatic.  The cut therefore increases by
two, again contradicting \emph{(E1)}.

The remaining case is $d_G(x)=2$. Consider the following recoloring: switch the color of $v$ from red to blue and switch the colors of $x$ and $w$ from blue to red.
This is locally feasible.  The new red vertices $x$ and $w$ support each other, $p$ and $q$ retain red neighbors, and $v$ become a singleton blue vertex, while removing $x,w$ from the blue
graph cannot violate the blue degree condition. 

Every neighbor of $w$ other than $x$ was red before the recoloring, because
$w$ already had the blue neighbor $x$. Hence the change in the size of the red--blue cut is $2-(d_G(w)-1)$ that is non-negative.
If $d_G(w)\le 2$ then the cut increases, contradicting \emph{(E1)}. If $d_G(w)=3$, then the size of the red--blue cut is unchanged but one red vertex is replaced by two red vertices, so $|\R|$
increases by one, contradicting \emph{(E2)}. Thus the third neighbor of $v$ cannot be blue, completing the proof.
\end{proof}

\begin{corollary}\label{cor:core-path-cycle}
Every component of the red core $\K$ is a path, a cycle, an isolated vertex, or
an edge.  Every internal vertex of a core path, and every vertex of
a core cycle has a private red leaf.
\end{corollary}

\begin{proof}
The first assertion follows from Lemma~\ref{lem:maxcut-normal-form}\emph{(i)}.  The
second follows from part \emph{(iii)}, since every vertex of a core cycle is
internal.
\end{proof}


The private red leaf at an internal core vertex provides a local recoloring that
breaks the core while preserving the extremal cut value.

\begin{lemma}[Neutral switch lemma]\label{lem:neutral-switches}
Let $v$ be an internal core vertex, let $p,q$ be its core-neighbors, and let
$x$ be its private red leaf.  Then $d_G(x)\in\{2,3\}$.  

Moreover, every
vertex $y\in N_G(x)\setminus\{v\}$ is an isolated blue vertex of degree three.  For every such $y$, interchanging the colors of $v$ and $y$ preserves local feasibility, $|E(\R,\B)|$, and $|\R|$.
\end{lemma}

\begin{proof}
First, suppose to the contrary that $d_G(x)=1$.  Recolor both $v$ and $x$ blue.  The edge
$vx$ becomes a blue edge, $v$ has no other blue neighbor, and $p,q$ retain
red neighbors.  Thus local feasibility is preserved.  The edges $vp$ and
$vq$ enter the red--blue cut, while $vx$ changes from red--red to blue--blue.
The cut therefore increases by two, contradicting \emph{(E1)}.  Hence
$d_G(x)\ge 2$.

Fix $y\in N_G(x)\setminus\{v\}$.  Because $x$ is a private red leaf, $y$ is
blue. 
Recolor $v$ blue and $y$ red. The leaf $x$ loses the red neighbor $v$ but
gains the red neighbor $y$; the new red vertex $y$ is therefore supported
by $x$.  Removing $y$ from the blue graph cannot violate the blue condition,
and the new blue vertex $v$ has the three red neighbors $p,q,x$.  Hence the
switch is locally feasible.

The three edges incident with $v$ enter the cut.  The edge $xy$ leaves the
cut.  Among the other $d_G(y)-1$ edges at $y$, exactly $d_{G[B]}(y)$ enter the cut and $d_G(y)-1-d_{G[B]}(y)$ leave it. Thus
\[
\begin{aligned}
  \Delta |E(\R,\B)|
  &=3-1+d_{G[B]}(y)-(d_{G}(y)-1-d_{G[B]}(y))\\
  &=3-d_{G}(y)+2d_{G[B]}(y).
\end{aligned}
\]
Since $d_{G}(y)\le 3$ and $d_{G[B]}(y)\ge 0$, this quantity is non-negative.  By \emph{(E1)}, it must
be zero.  Consequently $d_{G}(y)=3$ and $d_{G[B]}(y)=0$.  Therefore $y$ is a degree-3 vertex in $G$ and isolated in $G[\B]$, and the displayed recoloring is cut-neutral. It clearly preserves $|\R|$.
\end{proof}

The previous lemma determines those vertices where such a neutral switch can be performed.

\begin{definition}\label{def:switch}
Let $v$ be an internal core vertex with private red leaf $x$.  A vertex
$y\in N_G(x)\setminus\{v\}$ is called a \emph{switch vertex for $v$}, and
$(v,y)$ is called a \emph{switch candidate}. We might also refer to $v$ as the \emph{center} to the switch vertex $y$.  Performing the switch candidate
$(v,y)$ means recoloring $v$ blue and $y$ red.
\end{definition}

By Lemma~\ref{lem:neutral-switches}, every internal core vertex has at least
one switch candidate.

\begin{lemma}\label{lem:switch-private}
A switch vertex is adjacent to the private red leaf of at most one internal core
vertex.  In addition, if $y$ is a switch vertex for $v$ with private leaf
$x$, then the two neighbors of $y$ other than $x$ are red vertices that are
neither internal core vertices nor private red leaves of internal core
vertices.
\end{lemma}

\begin{proof}
Suppose that a switch vertex $y$ is adjacent to private red leaves $x_v$ and
$x_w$ belonging to two distinct internal core vertices $v$ and $w$.
Recolor $y$ red and recolor $v,w$ blue.

The leaves $x_v,x_w$ are now supported by $y$.  Every core-neighbor of $v$ or
$w$ loses at most one red neighbor, except possibly a red vertex adjacent to
both $v$ and $w$.  Such a common core-neighbor, if it exists, has degree two
in the core and is therefore internal. By
Lemma~\ref{lem:maxcut-normal-form}\emph{(iii)}, it has its own private red leaf neighbor and remains
non-isolated. The new blue vertices $v,w$ are either isolated or, if they are
adjacent, form a single blue edge again by Lemma~\ref{lem:maxcut-normal-form}\emph{(iii)}.  Thus the recoloring is locally feasible.

If $v$ and $w$ are non-adjacent, recoloring them blue adds six edges to the
red-blue cut.  If they are adjacent, their common edge changes from red--red to blue--blue and the other four incident edges enter the red-blue cut, so the increase from $v,w$ is four.  By Lemma~\ref{lem:neutral-switches}, $y$ is an isolated degree-3 blue vertex with three red neighbors, so recoloring $y$ red removes three edges from the red-blue cut.  The total cut change is therefore $+3$ or $+1$, in either
case contradicting \emph{(E1)}.  Hence a switch vertex cannot serve two distinct
internal core vertices.

Finally, Lemma~\ref{lem:maxcut-normal-form} says that every neighbor of an
internal core vertex is red, so a blue switch vertex cannot be adjacent to
an internal core vertex.  The first part of the proof excludes every private
red leaf of an internal core vertex other than its own leaf $x$.  This proves
the second assertion.
\end{proof}


The external red neighbors of a switch vertex can only lie near the ends of
old red-core paths.  We package those possible attachment points into
terminal regions.

\begin{definition}\label{def:terminal-region}
Let $C$ be a connected component of $G[\R]$.  Let $I(C)$ be the set of
vertices of $C$ that are internal in the red core, and for each $v\in I(C)$
let $x_v$ denote the private red leaf of $v$.  Define the \emph{terminal
region} of $C$ by
\[
  \T(C)
  =V(C)\setminus\bigcup_{v\in I(C)}\{v,x_v\}.
\]
For a nontrivial core path, $\T(C)$ is the union of the two endpoint stars.
For a core cycle it is empty.  If the core of $C$ has at most two vertices,
or if $C$ is a red $K_2$, the definition retains the corresponding small
component.  The two endpoint stars of one red component are deliberately
regarded as one terminal region.
\end{definition}

A blue vertex is said to be \emph{incident with} a terminal region if it has
a neighbor in that region.

\begin{lemma}\label{lem:terminal-bound}
Every terminal region is incident with at most eight blue vertices.
Furthermore, if $y$ is a switch vertex for an internal core vertex $v$ with
private red leaf $x_v$, then every neighbor of $y$ other than $x_v$ lies in a
terminal region. 
\end{lemma}

\begin{proof}
Consider first an endpoint $c$ of a core path.  Let $l(c)$ be the number of
red leaves adjacent to $c$.  Since $c$ has one core-neighbor and $G$ is
subcubic, $l(c)\in\{1,2\}$.  If $l(c)=1$, then $c$ has at most one blue
neighbor and its red leaf has at most two blue neighbors, so the endpoint
star is incident with at most three blue vertices.  If $l(c)=2$, then $c$
has no blue neighbor and the two red leaves have at most four blue neighbors
in total.  Hence each endpoint star is incident with at most four blue
vertices, and the terminal region of a core path is incident with at most
eight.

The remaining small-core cases satisfy the same bound directly.  If the core
is a single vertex, the whole component is incident with at most six blue
vertices.  If the core is a single edge, the two endpoint-star estimates give
at most eight.  A red $K_2$ is incident with at most four blue vertices.
A core cycle has empty terminal region.

For the second assertion, Lemma~\ref{lem:switch-private} excludes internal
core vertices and private leaves of internal core vertices from the two red
neighbors of $y$ other than $x_v$.  Every remaining red vertex belongs to a
terminal region by Definition~\ref{def:terminal-region}.
\end{proof}

We now define the conflict graph, which we later use in the simultaneous neutral switch selection.
Let $\mathcal B$ be any family of pairwise disjoint blocks of internal core
vertices.  For a block $B\in\mathcal B$, let
\[
  A_B=\{(v,y): v\in B \text{ and } y \text{ is a switch vertex for }v\}
\]
be the class of switch candidates available in $B$.

\begin{definition}\label{def:conflict}
The \emph{conflict graph} $H$ has the switch candidates in
$\bigcup_{B\in\mathcal B}A_B$ as vertices.  Distinct candidates $(v,y)$ and
$(v',y')$ are adjacent in $H$ if the blue vertices $y$ and $y'$ are
incident with the same terminal region.
\end{definition}

The numerical values of the next statement are tailor-made for our calculations.

\begin{lemma}\label{lem:conflict-degree}
If every block in $\mathcal B$ contains $28$ internal core vertices, then $|A_B|\ge 28$ for every $B\in\mathcal B$, and the conflict graph has maximum degree at most $14$.
\end{lemma}

\begin{proof}
Every internal core vertex has at least one switch candidate by
Lemma~\ref{lem:neutral-switches}, and by
Lemma~\ref{lem:switch-private} a switch vertex cannot belong to two distinct
internal core vertices. Hence each block of $28$ centers gives at least $28$ distinct
candidate vertices.

Fix a candidate $(v,y)$.  Besides the private leaf of $v$, the degree-3 switch
vertex $y$ has two red neighbors.  By Lemma~\ref{lem:terminal-bound}, they
lie in at most two terminal regions.  Each terminal region is incident with
at most eight blue vertices, hence with at most eight candidate switch
vertices.  Through a fixed terminal region, $(v,y)$ therefore conflicts with
at most seven other candidates.  Consequently $d_H((v,y))\le 7+7=14$.
\end{proof}


We next place blocks of $28$ internal core vertices.  The placement is chosen
only to make the final path count transparent.

\begin{definition}\label{def:block-placement}
For each component of the red core, place blocks as follows.
\begin{enumerate}[label=(B\arabic*)]
  \item Let $P$ be a core path with $n$ internal vertices.
  \begin{enumerate}[label=(\alph*)]
    \item If $n\le 27$, place no block.
    \item If $28\le n\le 55$, place one block of $28$ consecutive internal
    vertices so that the numbers of uncovered internal vertices at the two
    ends differ by at most one.
    \item If $n\ge 56$, write
      $n=28t+r,
     0\le r<28,
     t\ge 2$.
    
    Place $t$ disjoint consecutive $28$-blocks so that the first block begins
    at the left end of the internal-vertex sequence, the last block ends at
    the right end, and the $r$ uncovered internal vertices form one central
    run.
  \end{enumerate}

  \item Let $Q$ be a core cycle on $m$ vertices.  If $m\le 87$, place no
  block.  If $m\ge 88$, write $m=28t+r$ with $0\le r<28$ and place $t$
  consecutive $28$-blocks around the cycle, leaving one uncovered run of
  $r$ vertices.
\end{enumerate}
The family of these fixed blocks is denoted by $\mathcal B$.
\end{definition}

Now, we prove that a non-conflicting neutral switch selection is available.

\begin{lemma}\label{lem:independent-selection}
There is a choice of one switch candidate from every block such that no two
chosen switch vertices are incident with the same terminal region.
Moreover, among the chosen centers no three consecutive vertices occur on a
core path or core cycle.
\end{lemma}

\begin{proof}
By Lemma~\ref{lem:conflict-degree}, we know that $|A_B|\ge 28$ for every $B\in \mathcal B$, and $\Delta(H)\le 14$.
Since $28=2\cdot14$, we can apply, Haxell's result, Theorem~\ref{thm:Haxell} that gives an independent transversal in $H$.
By Definition~\ref{def:conflict}, its independence means precisely that no two chosen switch vertices are incident with the same terminal region.

Each block contributes only one chosen center and has $28$ consecutive core
vertices.  Therefore three consecutive core vertices cannot all be chosen:
any three consecutive vertices meet at most two blocks, while each block
contributes at most one center.
\end{proof}

Now, we show that if we perform these switches simultaneously then the new coloring remains locally feasible.

\begin{lemma}\label{lem:simultaneous-feasibility}
Perform simultaneously all switches selected in
Lemma~\ref{lem:independent-selection}.  The resulting coloring is locally
feasible.
\end{lemma}

\begin{proof}
Let $S$ be the set of selected internal core vertices and let $Y$ be the set
of selected switch vertices.  Recolor every vertex of $S$ blue and every
vertex of $Y$ red.

We first check the blue condition.  By
Lemma~\ref{lem:maxcut-normal-form}, a selected center has only red neighbors
before the recoloring: its two core-neighbors and its private red leaf.  The
only possible blue neighbors it acquires are therefore selected adjacent
core vertices.  By Lemma~\ref{lem:independent-selection}, no three
consecutive core vertices are selected, so every selected center has at most
one selected core-neighbor.  Every selected switch vertex is removed from
the blue graph.  No other blue vertex gains a new blue neighbor.  Hence the
new blue graph has maximum degree at most one.

For the red condition, let $v\in S$ have private red leaf $x_v$ and selected
switch vertex $y_v\in Y$.  The leaf $x_v$ loses the red neighbor $v$ but gains the
red neighbor $y_v$, and $y_v$ is supported by $x_v$.  An unselected internal
core vertex that loses one or even both core-neighbors still has its private
red leaf.  A core endpoint that loses its unique core-neighbor still has at
least one red leaf, because it belongs to the red core.  All other old red
vertices retain a red neighbor.  Thus the new red graph has no isolated
vertices.
\end{proof}


We now show that the same simultaneous switches force all red paths to have
at most $88$ edges.

\begin{lemma}\label{lem:one-switch} 
After the simultaneous recoloring, every red component contains at most one
selected switch vertex.
\end{lemma}

\begin{proof}
Let $y$ be a selected switch vertex for a selected center $v$, with private
leaf $x_v$.  The branch through $x_v$ cannot reach the old red core because
$v$ is now blue.  By Lemma~\ref{lem:terminal-bound}, every other red neighbor
of $y$ lies in a terminal region of an old red component.

Suppose a red component of the new coloring contained another selected
switch vertex $y'$.  Any red path from $y$ to $y'$ would have to leave $y$
through a terminal region of some old red component and eventually enter
$y'$ through a terminal region.  Along such a chain, some old terminal region
would be incident with two selected switch vertices.  This is excluded by
Lemma~\ref{lem:independent-selection}.  Hence no new red component contains
two selected switch vertices.
\end{proof}

\begin{lemma}\label{lem:path-accounting} 
After the simultaneous recoloring of
Lemma~\ref{lem:simultaneous-feasibility}, every red path has at most $88$
edges.
\end{lemma}

\begin{proof}
We first record the maximum lengths of pieces inherited from an old red
component.

Let a core path have $n$ internal core vertices.

If $n\le 27$, the path receives no block.  Its old red component has diameter
at most $n+3\le 30$. 
Indeed, the core has $n+2$ vertices and $n+1$ core edges, and a longest red
path can extend by at most one red-leaf edge at each end.

Suppose $28\le n\le 55$.  The block in the middle leaves at most $14$ uncovered
internal vertices at either end.  Wherever the selected center lies inside
the $28$-block, at most $14+27=41$
internal vertices lie between that center and either end of the core.  After
the selected center is recolored blue, every terminal red leg therefore has
length at most $43$ edges, including a possible red-leaf edge at each end of
the leg.

Suppose $n\ge 56$.  The first and last blocks start at the two ends of the
internal-vertex sequence.  Hence a selected center in an end block has at
most $27$ internal vertices between it and the corresponding core endpoint.
Every terminal red leg consequently has length at most $29$ edges.
Between two consecutive selected centers, the largest possible run of
unselected internal vertices consists of a suffix of one $28$-block, the
unique central uncovered run, and a prefix of the next $28$-block.  Its size
is therefore at most $27+27+27=81$.
Such an internal red piece contains no red path longer than $82$ edges, even
allowing a private red leaf at each end.

For a core cycle that receives blocks, the same argument gives runs of at
most $81$ unselected core vertices and hence red paths of at most $82$
edges.  An untouched core cycle has at most $87$ core vertices.  A simple
path can use at most $86$ cycle edges and can extend by one private red-leaf
edge at each end, for a total of at most $88$ edges.

Now consider an arbitrary red component after all switches.  If it contains
no selected switch vertex, it is contained in one of the pieces just bounded,
so every red path in it has at most $88$ edges.

If it contains a selected switch vertex $y$, then by
Lemma~\ref{lem:one-switch} it contains no other selected switch vertex.  The
private red leaf branch of $y$ has length one.  Each of the other at most two
branches of $y$ enters a terminal region of an old red component.  The
maximum distance available inside such a terminal piece is at most $43$:
this is the bound for a medium blocked core path, while the short and long
paths have smaller terminal bounds, $30$ and $29$, respectively.
A core cycle has empty terminal region and hence contributes no such branch.
Thus a longest red path using $y$ has at most $43+1+1+43=88$
edges.  This proves the lemma.
\end{proof}


Let us recall the main result of the paper which follows now from the previous observations.

\medskip

\noindent{\bf Theorem~\ref{thm:main88}.}\
{\it Every subcubic graph $G$ has an $88$-crumby coloring. In other words, there exists a red--blue coloring such that
$  \Delta(G[\B])\le 1,
  \delta(G[\R])\ge 1,
$
and every path in $G[\R]$ has at most $88$ edges, where $R$ is the set of red and $B$ is the set of blue vertices.
}

\begin{proof}
By Lemma~\ref{lem:existence-local}, locally feasible colorings exist.  Choose
an extremal one.  Lemma~\ref{lem:maxcut-normal-form} gives the path--cycle
structure of the red core and a private red leaf at every internal core
vertex.  Lemma~\ref{lem:neutral-switches} supplies at least one neutral switch
candidate at each internal core vertex, while
Lemmas~\ref{lem:switch-private} and \ref{lem:terminal-bound} implies that all switch vertices can only have red neighbors from terminal regions and each terminal region can have a bounded number of such incidences.

Place the blocks like in Definition~\ref{def:block-placement}.  By
Lemma~\ref{lem:conflict-degree} and Haxell's theorem, choose one
non-conflicting switch candidate from each block as in
Lemma~\ref{lem:independent-selection}. Perform all selected switches
simultaneously.  Local feasibility follows from
Lemma~\ref{lem:simultaneous-feasibility}, and the red-path bound follows from
Lemma~\ref{lem:path-accounting}.  Therefore every red path has at most $88$
edges.
\end{proof}

\begin{rem}
    The constant in Theorem~\ref{thm:main88} is not intended to be optimal. A further improvement is possible by encoding each switch candidate as an edge between the terminal regions incident with its two external neighbors and applying the corollary of Theorem 1.5 in a paper by Aharoni, Berger and Meshulam\cite{abm} instead of Haxell's independent transversal theorem. Since the resulting multigraph has maximum degree at most 8, blocks of 16 vertices suffice; the corresponding path count yields an upper bound of 52 edges. We have retained the Haxell formulation because it keeps the switch-selection argument in the simpler conflict-graph framework.
\end{rem}

\section{Discussion}

Besides Conjecture~\ref{conj}, there are natural questions that are left to answer in both topics. We list these problems to finish our paper.

\begin{problem}
    For any cubic graph $G$ that admits a perfect matching what is the minimum number $k$ for which there always exists a red--blue coloring of $V(G)$ such that the maximum order of a monochromatic component is at most $k$ while there are no monochromatic singletons?
\end{problem}

We proved that $k=12$ is enough but so far the lower bound on $k$ is $7$.

\begin{problem}
    Improve similarly the bound on $k$ for cubic graphs without a perfect matching.
\end{problem}

About the generalized crumby colorings, the most interesting question is how small $\ell$ can be that guarantees an $\ell$-crumby coloring for any subcubic graph. 

\begin{problem}
    Prove the existence of an $\ell$-crumby coloring for any subcubic graph for some $\ell<52$. Is it true that $\ell=4$ is enough?
\end{problem}


\section*{Acknowledgement}
This work was initiated in a private discussion between the first author and Carsten Thomassen a decade ago. The hospitality of the Technical University of Denmark and the generosity of Carsten Thomassen are gratefully acknowledged.

During the preparation of this work, the authors used ChatGPT 5.6 Sol  to improve the clarity of the proofs, in particular finding Proposition~\ref{prop:hea}. It also pointed us to the Berke-Szabó paper. Following the authorship guidelines of COPE and arXiv, this AI tool was used strictly as an assistant and is not listed as a co-author. The authors maintain sole responsibility for the originality, accuracy, validity, and integrity of the entire content, data, and conclusions presented in this manuscript.

The first author is partially supported by ERC Advanced Grants ``GeoScape" No.88271. and “ERMiD”, no. 101054936, and NKFIH Grant K.131529. The second author was supported at the initial stage of this work by the EKÖP-24-4-SZTE-609 Program which belongs to the Ministry for Culture and Innovation from the source of the National Research, Development and Innovation Fund. The second author is also supported by the National Research, Development and Innovation Fund of the Ministry for Innovation and Technology of Hungary under grant SNN 152582.

\end{document}